\documentclass{amsart}
\usepackage{amsthm}
\usepackage{amssymb}
\usepackage{colonequals}
\usepackage{dsfont}
\usepackage{enumerate}

\usepackage{tikz-cd}
\tikzset{labl/.style={anchor=south, rotate=270, inner sep=.5mm}}

\usepackage{hyperref}
\hypersetup{%
  bookmarksnumbered=true,%
  colorlinks=true,%
  linkcolor=blue,%
  citecolor=blue,%
  filecolor=blue,%
  menucolor=blue,%
  urlcolor=blue,%
  bookmarksopen=true,%
  bookmarksdepth=2,%
  pageanchor=true}

\makeatletter
\@namedef{subjclassname@2020}{\textup{2020} Mathematics Subject Classification}
\makeatother

\numberwithin{equation}{section}

\theoremstyle{plain}

\newtheorem{theorem}[equation]{Theorem}

\newtheorem{lemma}[equation]{Lemma}

\theoremstyle{definition}

\newtheorem{example}[equation]{Example}

\theoremstyle{remark}
\newtheorem*{ack}{Acknowledgements}

\newcommand{\cat}[1]{\mathcal{#1}}
\newcommand{\dcat}[1]{\operatorname{D}(\operatorname{Mod}#1)}
\newcommand{\ext}{\operatorname{Ext}}
\newcommand{\fHom}{\operatorname{\mathcal{H}\!\!\;\mathit{om}}}
\newcommand{\gam}{\varGamma} 
\newcommand{\rgam}{\mathbf{R}\varGamma}
\newcommand{\hh}{\operatorname{H}}
\newcommand{\Hom}{\operatorname{Hom}}
\newcommand{\rhom}{\operatorname{RHom}}
\newcommand{\rmod}{\operatorname{mod}}

\newcommand{\lotimes}{\otimes^{\mathbf L}}
\newcommand{\one}{\mathds{1}}
\newcommand{\perf}{\operatorname{Perf}}
\newcommand{\proj}{\operatorname{Proj}}
\newcommand{\rad}{\operatorname{rad}}
\newcommand{\res}{\operatorname{res}}
\newcommand{\spec}{\operatorname{Spec}}
\newcommand{\stmod}{\operatorname{stmod}}
\newcommand{\StMod}{\operatorname{StMod}}
\newcommand{\thick}{\operatorname{Thick}}
\newcommand{\uHom}{\underline{\Hom}}
\newcommand{\uEnd}{\underline{\operatorname{End}}}

\newcommand{\fa}{\mathfrak a}
\newcommand{\fm}{\mathfrak m}

\newcommand{\fp}{\mathfrak p}
\newcommand{\fq}{\mathfrak q}

\title[Locally dualizable modules]
{Locally dualizable modules abound}
\author[Jon F. Carlson]{Jon F. Carlson}
\address{Department of Mathematics, University of Georgia, 
Athens, GA 30602, USA}
\email{jfc@math.uga.edu}

\author[S.~B.~Iyengar]{Srikanth B. Iyengar}
\address{Department of Mathematics, University of Utah,
Salt Lake City, UT 84112, USA}
\email{srikanth.b.iyengar@utah.edu}

\date\today



\begin{document}

\begin{abstract}
It is proved that given any prime ideal $\mathfrak{p}$ of height at least 2 in a countable commutative noetherian ring $A$, there are uncountably many more 
dualizable objects in the $\mathfrak{p}$-local $\fp$-torsion stratum of the derived category of $A$ than those that are obtained as retracts of images of perfect $A$-complexes. An analogous result is established dealing with the stable module category of the group algebra,  over a countable field of positive characteristic 
$p$, of an elementary abelian $p$-group of rank at least 3.
\end{abstract}

\maketitle

\section{Introduction}
This work is about dualizable objects in tensor triangulated 
categories arising in commutative algebra and the modular 
representation theory of finite groups. An object $D$ in a  
tensor triangulated category $\cat{C}$ is \emph{dualizable} 
if the natural map 
\[
\fHom(D,\one)\otimes X \longrightarrow \fHom(D,X)
\]
is an isomorphism for all $X$ in $\cat{C}$. Here $\fHom$ and 
$\otimes$ are the internal function object and product in 
$\cat C$, respectively, and $\one$ is the unit of the product.

Consider $\dcat A$, the derived category of a commutative noetherian 
ring $A$, with tensor structure given by the derived tensor product 
$-\lotimes-$, the unit is $A$, and function object is $\rhom_A(-,-)$. 
The dualizable objects in $\dcat A$ are precisely the perfect 
$A$-complexes, $\perf A$. These are also the compact objects, and hence 
$\dcat A$ is rigid as a compactly generated tensor triangulated category.

We consider also $\StMod(kG)$ the stable module category of 
a finite group $G$, with $k$ a field of positive characteristic 
dividing $|G|$. In this case the tensor structure is given by 
$-\otimes_k-$ with diagonal $G$ action and the function object 
is $\Hom_k(-,-)$, again with diagonal $G$-action; the unit is $k$ 
with trivial $G$-action. The dualizable objects are those in 
$\stmod(kG)$, namely, the $kG$ modules that are stably isomorphic 
to finite dimensional ones, and hence coincide with the compact 
objects, so $\StMod(kG)$ is also rigid.

Both $\dcat A$ and $\StMod(kG)$ admit natural stratifications 
into ``local" triangulated subcategories that determine, to a 
large extent, their global structure. In $\dcat A$ the strata 
are parameterized by points $\fp\in \spec A$ and the stratum 
corresponding to $\fp$ consists of the $\fp$-local and $\fp$-torsion 
complexes in $\dcat A$, which is denoted $\gam_{\fp}\dcat A$; 
see Section~\ref{se:localalgebra} for details. There is an 
analogous stratification $\StMod(kG)$, with parameter space 
$\proj H^*(G,k)$; see Section~\ref{se:groups}. In both cases, 
the strata are again tensor triangulated subcategories, and it is 
of interest to understand the dualizable objects in these categories. 
A noteworthy feature now is that, unless $\fp$ is minimal, there are 
many more dualizable objects than compact ones, so the strata are not rigid. 

There is a natural functor $\gam_\fp \colon \dcat A\to \gam_\fp \dcat A$ and  the dualizable objects in $\gam_\fp\dcat A$ are generated as a thick subcategory by $\gam_\fp A$; see \cite[Theorem]{Benson/Iyengar/Krause/Pevtsova:2023a}. The question arose whether $\gam_\fp\perf A$ is dense in the subcategory of local dualizable objects; in other words, whether each dualizable object in $\gam_\fp\dcat A$ is a retract of a complex  $\gam_{\fp}P$, with $P$ a perfect $A$-complex. The point of this paper is that when $A$ is countable, there are uncountably many, mutually non-isomorphic, indecomposable dualizable objects in $\gam_\fp\dcat A$, but at most countably many that are retracts of images of perfect complexes in $A$; see Theorem~\ref{th:ca}. 

There is an analogous description of the dualizable objects in $\gam_\fp\StMod(kG)$, established in \cite{Benson/Iyengar/Krause/Pevtsova:2024a}, and once again it turns that there can be many more dualizable objects in this strata than direct summands of those induced from $\stmod(kG)$, the global dualizable objects; see Theorem~\ref{th:groups}. 

\begin{ack}
This work is partly supported by National Science Foundation Grant DMS-200985 (SBI).
\end{ack}

\section{Local algebra}
\label{se:localalgebra}

Let $A$ be a commutative noetherian ring and $\dcat A$ the (full) derived category of $A$-modules, viewed as a tensor triangulated category, where the product of $A$-complexes $X,Y$ is the derived tensor product $X\lotimes_AY$. The derived category is rigidly compactly generated, with compact objects the perfect $A$-complexes, that is to say, those that are isomorphic, in $\dcat A$, to bounded complexes of finitely generated projective modules. We denote this category $\perf A$. It is the thick subcategory of $\dcat A$ generated by $A$.

Given a point $\fp \subseteq \spec R$ consider the exact functor
\[
\gam_{\fp} \colon \dcat A\to \dcat A\quad\text{where $X \mapsto \rgam_{V(\fp)}(X_\fp)$}
\]
where $\mathbf{R}\gam_{V(\fp)}(-)$ is the functor representing local cohomology with support in the ideal $\fp$ of $A$. The image $\gam_\fp\dcat A$ consists of precisely the $\fp$-local and $\fp$-torsion complexes. It is a tensor triangulated category in its own right, with product induced from that on $\dcat A$, unit $\gam_\fp A$, and function object $\gam_{\fp}\rhom_A(-,-)$. It is not rigid, unless $\fp$ is minimal, and there are many more rigid objects than compact ones. See \cite[Section 4]{Benson/Iyengar/Krause/Pevtsova:2023a} for proofs of these assertions.

The main result of this section is as follows.

\begin{theorem}
\label{th:ca}
Let $A$ be a countable commutative noetherian ring and $\fp\in \spec A$ such that $\mathrm{height}\, \fp \ge 2$. There exist uncountably many mutually non-isomorphic indecomposable dualizable objects in $\gam_\fp\dcat A$ none of which is a retract of an object in $\gam_\fp \perf A$.
\end{theorem}

In fact the argument, which is given towards the end of this section, 
shows that $\gam_\fp\perf A$ contains only countably many isomorphism 
classes of objects, even allowing for retracts, but that there are 
uncountably many non-isomorphic indecomposable objects $\thick(\gam_\fp A)$. 
It is easy to check that the latter subcategory consists of dualizable 
objects in $\gam_{\fp}\dcat A$. As it happens, these are all the dualizable 
objects. This is the main result in \cite{Benson/Iyengar/Krause/Pevtsova:2023a},
but we do not need this fact here.

We record a simple observation.

\begin{lemma}
\label{le:countable}
When $A$ is a countable noetherian ring, there are only countably many  
isomorphism classes of objects in $\rmod A$, and, more generally, 
only countably many isomorphism classes in $\operatorname{D}^b(\rmod A)$.
\end{lemma}

\begin{proof}
Any finitely generated $A$-module occurs as a cokernel of a map $A^m\to A^n$, for each nonnegative integers $m,n$, and each such map is given by a matrix of size $m\times n$ with coefficients  in $A$. Since $A$ is countable, there are only countably many such matrices, which justifies the claim about $\rmod A$. The one about complexes follows because in $\dcat A$ any complex with finitely generated homology is isomorphic  to one of the form
\[
0\longrightarrow M_n\longrightarrow F_{n-1}\longrightarrow \cdots \longrightarrow F_1 \longrightarrow F_0 \longrightarrow 0\,,
\]
where $M_n$ is in $\rmod R$ and each $F_i$ is a finite free $A$-module.
\end{proof}

The result below is also well-known; see \cite[p.~500]{Hassler/Wiegand:2009}. 

\begin{lemma}
\label{le:uncountable}
Let $A$ be a commutative noetherian local ring that is complete with 
respect to $\fm$, its maximal ideal. If $\dim A\ge 2$, then there 
exists an uncountable collection of distinct prime ideals in $A$, 
each of height one. In the same vein, there is an uncountable 
collection of elements $\{a_u\}_{u\in U}$ with 
$\sqrt{a_u}\ne \sqrt{a_v}$ for $u\ne v$.
\end{lemma}

\begin{proof}
Assume to the contrary that there are  only countable many prime 
ideals $\{\fp_i\}_{i\geqslant 1}$ of height one. Since each 
non-invertible element in $A$ is contained in a height one prime, 
by Krull's Principal Ideal theorem, it follows that 
$\fm\subseteq \cup_{i}\fp_i$.  Since $A$ is complete, it has the 
countable prime avoidance property; see, for instance, 
\cite{Burch:1972, Sharp/Vamos:1985}. We deduce that $\fm\subseteq \fp_i$ 
for some $i$, contradicting the hypothesis that $\dim A\ge 2$.

Because every prime ideal of height one is minimal over a principal 
ideal, and the radical of principal ideal is a finite intersection of 
prime ideals, the second part of the assertion follows from the first.
\end{proof}

\begin{proof}[Proof of Theorem~\ref{th:ca}]
One has that $\gam_{\fp}\dcat A\simeq \gam_\fp\dcat{A_\fp}$, so replacing $A$ by its localization at $\fp$ we can suppose it is local, say with maximal ideal $\fm$, with $\dim A\ge 2$. Let $\widehat A$ be the $\fm$-adic completion of $A$. A key input in the arguments presented below is that the natural map of rings $A\to \ext_A^*(\rgam_{\fm}A,\rgam_{\fm}A)$ factors through the completion map $A\to \widehat A$ and yields an isomorphism 
\[
\widehat A\xrightarrow{\ \cong\ } \ext_A^*(\rgam_{\fm}A,\rgam_{\fm}A).
\]
Hence, derived Morita theory yields the Greenlees-May adjoint equivalence
\[
\begin{tikzcd}[column sep=large]
\thick_{\widehat A}(\widehat A) \arrow[leftarrow,yshift=-1ex,swap,rr, "{\rhom_A(\rgam_\fm A,-)}"]
 	&& \arrow[leftarrow,yshift= 1ex,ll,swap, "{\rgam_{\fm}A\lotimes_{\widehat A}-}"] \thick_A(\rgam_\fm A) \,.
	\end{tikzcd}
\]
It is also helpful that $\rhom_A(\rgam_\fm,-)\cong \mathbf{L}\varLambda^\fm(-)$, the left derived functor of $\fm$-adic completion.  See the discussion around \cite[(4.2), (4.3)]{Benson/Iyengar/Krause/Pevtsova:2023a}. 

Let $\{\fp_u\}_u$ be the uncountable collection of prime ideals in $\widehat R$ supplied by Lemma~\ref{le:uncountable}. For each $\fp_u$ let $K_u$ denote the Koszul complex on $\rgam_\fm R$ on a minimal generating set for the ideal $\fp_u$ of $\widehat R$. The complex $K_u$ has the following properties:
\begin{enumerate}[\quad\rm(1)]
\item
Each $K_u$ is  dualizable in $\gam_{\fm}\dcat A$;
\item
One has $K_u\not\cong K_v$ for $u\ne v$;
\item
Each $K_u$ is indecomposable.
\end{enumerate}
Indeed (1) holds because $\rgam_\fm A$, being the unit of the product on $\gam_\fm\dcat A$ is dualizable, and $K_u$ is finitely built from $\rgam_\fm A$.

Under the adjoint equivalence above, $\rgam_\fm A$ is mapped to $\mathbf{L}\varLambda^\fm A$, so  $K_u$ is mapped to the Koszul complex on $\widehat A$ on the chosen minimal generating set for $\fp_u$. It follows that the kernel $I_u$ of the natural map 
\[
\widehat A \longrightarrow \ext^*_{A}(K_u,K_u) \cong \ext^*_{\widehat A}(\mathbf{L}\varLambda^\fm K_u,\mathbf{L}\varLambda^\fm K_u)
\]
satisfies $\sqrt{I_u}=\fp_u$. Since the $\{\fp_u\}_u$ are distinct, (2) 
follows.  Moreover, the Koszul complex over $\widehat A$ on any minimal 
generating set for the ideal $\fp$ is indecomposable in $\dcat{\widehat A}$, 
by \cite[Proposition~4.7]{Altmann-et-all:2017}. Consequently, $K_u$ 
is indecomposable in $\dcat{A}$. This justifies (3).
  
Since the collection $\{K_u\}_{u}$ is uncountable, to complete the proof we have to verify that there are only finitely many isomorphism classes of indecomposable direct summands of $\gam_\fm P$, with $P$ a perfect $A$-complex. To see this, note that since $\widehat A$ is complete, $\thick_{\widehat A}(\widehat A)$ is a Krull-Schmidt category, and hence so is $\thick_A(\rgam_\fm A)$, by the equivalence above. It remains to recall that there are only countably many isomorphism classes of perfect $A$-complexes, by Lemma~\ref{le:countable}.
\end{proof}

\section{Finite groups}
\label{se:groups}
Let $G$ be a finite group and $k$ a field of positive characteristic $p$, 
where $p$ divides the order of $G$. The stable category $\StMod(kG)$ of 
$kG$-modules modulo projective modules is a tensor triangulated category, 
there the product is $-\otimes_k-$ with the diagonal $G$-action, and  
the unit is the trivial $kG$-module $k$. The compact objects are the 
modules equivalent to finitely generated modules. 
They form a thick subcategory denoted $\stmod(kG)$. 

The cohomology ring $\hh^*(G,k)$ is a finitely generated graded $k$-algebra and $\ext^*_{kG}(M,N)$ is a finitely generated module over $\hh^*(G,k)$, for all $M,N$ in $\rmod kG$. The projectivized spectrum $V_G(k) = \proj \hh^*(G,k)$ is the collection of homogeneous prime ideals in $H^*(G,k)$, except the the maximal one, $H^{\geqslant1}(G,k)$.

The support variety $V_G(M)$ of a finitely generated module $M$ is  defined as the collection of those ideals that contain the  annihilator of $\ext^*_{kG}(M,M)$ in $\hh^*(G,k)$. Support varieties for infinite dimensional $kG$-modules are introduced in   \cite{Benson/Carlson/Rickard:1996}; see also \cite{Benson/Iyengar/Krause:2008}. These are subsets of $V_G(k)$ that are not  necessarily closed.

As in the previous section, for each $\fp$ in $V_G(k)$, there exists an exact functor
\[
\gam_{\fp} \colon \StMod(kG) \longrightarrow \StMod(kG) 
\]
whose image is the subcategory consisting of all $kG$-modules whose support 
is contained in $V$.  Then, $\gam_\fp\StMod(kG)$ is again tensor triangulated, 
with tensor product inherited from $\StMod(kG)$. The unit is $\gam_\fp k$,
and the function object is $\gam_\fp\Hom_k(-,-)$. 
There are numerous equivalent ways to characterize dualizable modules 
in this category; see \cite{Benson/Iyengar/Krause/Pevtsova:2024a}.  
Most importantly, the full subcategory of dualizable modules in 
$\gam_\fp\StMod(kG)$ form a thick triangulated  subcategory, and 
$\gam_\fp M$ is dualizable for each $M$ in $\stmod(kG)$.  The theorem 
below and its proof are similar to Theorem~\ref{th:ca}.

\begin{theorem}
\label{th:groups}
Let $k$ be a countable field of characteristic $p$, and $G$ an elementary abelian $p$-group  of rank $\geq 3$. Let $\fp$ be a closed point in $V_G(k)$. There exists an uncountable collection of mutually non-isomorphic, indecomposable dualizable modules in $\gam_\fp\StMod(kG)$, none of which is  a direct summand of $\gam_\fp M$ for  $M$ in $\stmod(kG)$.
\end{theorem}

\begin{proof}
Since the dualizable objects form a thick subcategory, it suffices to prove that there is an uncountable collection of mutually non-isomorphic, indecomposable objects in the thick category generated by $\gam_{\fp}\stmod(kG)$ that are not retracts of the images of the finite dimensional ones. This has nothing to do with tensor triangulated structure on $\gam_\fp \StMod(kG)$, and we are free to  choose any coalgebra structure on $kG$ that is convenient. 

We may assume $G = H \times \langle z \rangle$ where $H$ is  elementary abelian of rank $r-1$ and $z^p = 1$. That is, we can assume that the ideal $\fp$ is the radical of the restriction to a subalgebra $k[z]/(z^p) \in kG$, the inclusion  into $kG$  being a $\pi$-point associated to $\fp$ in the language of \cite{Friedslander/Pevtsova:2007}.  The choice of the complementary subalgebra $kH$ is somewhat arbitrary.  That is, we can choose $kH$ to be the subalgebra generated by any  collection $x_1, \dots, x_{r-1}$ in $\rad(kG)$ such that the images of  $z, x_1, \dots, x_{r-1}$ in $\rad(kG)/\rad^2(kG)$ form a $k$-basis.  Then $kG \cong kC \otimes kH$ where $C$ is generated by the unit $1+z$ and  $H$ is generated by $1+x_1, \dots, 1+x_{r-1}$. This is an isomorphism of $k$-algebras, but not generally as Hopf algebras. 

Keeping in mind that $\gam_\fp = \gam_{V(\fp)}$, from \cite[Proposition 5.2]{Carlson:2023a} we see that
\[
R\colonequals \uEnd_{kG}(\gam_\fp k) =\uHom_{kG}(\gam_\fp k,\gam_\fp k) 
	\cong \uHom_{kG}(\gam_\fp k,k) \cong \prod_{i=0}^\infty \hh^i(H,k)\,,
\]
as an additive group. The product is the obvious one, except that, if $p$ is odd, then any two elements of odd degree multiply to zero. 
Hence the endomorphism ring is a commutative local ring that, modulo a nilpotent ideal, is the completion of a polynomial ring of degree $r-1$. In particular its Krull dimension is $r-1\ge 2$.

For each $\zeta \in R$, we define the module  $K_\zeta$ to be the third object in the triangle 
\[
\begin{tikzcd}
K_\zeta \arrow[r] & \gam_\fp k \arrow[r,"\zeta"] & \gam_\fp k \arrow[r] & \,.
\end{tikzcd}
\]
Evidently, $K_\zeta$ is in the thick subcategory generated by 
$\gam_\fp\stmod(kG)$. Let $\fa_\fp(K_{\zeta})$ be the annihilator of
the $R$-module $\uHom_{kG}(\gam_\fp k,K_\zeta))$. Then we recall  
\cite[Theorem 7.6]{Carlson:2023a}  that the radical of 
$\fa_\fp (K_{\zeta})$ coincides with the radical of the ideal 
generated by $\zeta$. Thus, we have that  if $\zeta$ and $\gamma$ 
are elements in $R$ that generate ideals with different radicals, 
then $K_\zeta$ is not  isomorphic to $K_\gamma$. 

When $M$ is a finite dimensional module, the $R$-module $\uEnd_{kG}(\gam_\fp M)$ is finitely generated. Consequently, $\gam_\fp M$ has only a finite number of indecomposable summands, as otherwise,  $\uEnd_{kG}(\gam_\fp M)$ would have an infinite number of idempotents. It follows from the Lemma~\ref{le:countable} that the collection  of modules $K_\zeta$ that can be direct summands of modules of the  form $\gam_VM$ for $M$ finitely generated, is countable.  Since $\dim R\ge 2$ it remains to recall from Lemma~\ref{le:uncountable} that it has an  uncountable number of elements $\zeta$ having mutually distinct radicals. 
\end{proof}

\begin{example}
Suppose that $p=2$ and that $G$ is elementary abelian of order $8$.  
We use the notation of the previous proof.  We write $kG = kH \otimes kC$ 
where $kH = k[x,y]/(x^2,y^2)$ and $kC = k[z]/(z^2)$. Here the variety 
$V$ is the point corresponding to the inclusion $kC \hookrightarrow kG$. 
Choose $\zeta \in \uHom_{kG}(\gam_Vk, \gam_Vk)$ to have the  form 
$\zeta = (0, \zeta_1, \zeta_2, \dots ) $ where $\zeta_i \in \hh^i(H,k)$.
Assume that $\zeta_1 \neq 0$. Because $\zeta_1 \neq 0$, we have that 
the  restriction to $kH$ of $K_{\zeta}$ is a direct sum 
$\sum_{i=0}^\infty U_i$ where $U_i \cong kH$. Choose $u_i \in U_i$ 
to be a $kH$-generator,  for all $i$. With some calculation, it can be 
shown that the action of $z$ is give by a formula
\[
zu_i = \sum_{j = 0}^i (\alpha_jx +\beta_jy)u_{i-j}
\]
where for each $j$, the elements $\alpha_j, \beta_j \in k$, depend on 
the  choices of $\zeta_\ell$ for $\ell \leq j$. With some slight 
adjustment in the proof, Theorem~\ref{th:groups} tells us that 
if $k$ is countable,  then there is an uncountable collection of 
such elements $\zeta$ such that the resulting modules $M_\zeta$ 
are mutually non-isomorphic and not isomorphic to a direct summand of 
any $\gam_VM$ for any $M \in \stmod(kG)$.
\end{example}

\subsection*{General finite groups}

We end with the following result, extending Theorem \ref{th:groups} to any finite group.

\begin{theorem}
\label{th:groups-gen}
Let $k$ be a countable field and $G$ a finite group. Suppose that $\fp$ be a closed point in $V_G(k)$
that is contained in $\res^*_{G,E}(V_G(k))$ for some elementary abelian 
$p$-subgroup $E$ having rank $\geq 3$. There exists an uncountable 
collection of mutually non-isomorphic, indecomposable dualizable modules 
in $\gam_\fp\StMod(kG)$, none of which is  a direct summand of 
$\gam_\fp M$ for  $M$ in $\stmod(kG)$.
\end{theorem}

\begin{proof}
We use the induction functor $\StMod(kE) \to \StMod(kG)$ that takes a 
$kE$-module $M$ to $M^{\uparrow G} = kG \otimes_{kE} M$.
The restriction to a $kE$-module  of the idempotent 
module $\gam_\fp k$ is still an idempotent
module, and a support variety argument establishes that, in the 
stable category, it has the form 
\[
(\gam_\fp k)_{\downarrow E} \cong \sum \gam_\fq k
\]
where the sum is over the finite finite collection of closed points 
$\fq \in V_E(k)$ such that $\res_{G,E}^*(\fq) = \fp$. 
Then by Frobenius reciprocity, we have that 
\[
\gam_\fp k \otimes (k_{\downarrow E})^{\uparrow G} \cong 
((\gam_\fp k)_{\downarrow E})^{\uparrow G} \cong 
\sum (\gam_\fq k)^{\uparrow G}.
\]
As a consequence, the modules $(\gam_\fq k)^{\uparrow G}$ are dualizable. 

Let $\fq$ denote any one of the points with $\res_{G,E}^*(\fq) = \fp$.
Because the induction functor is exact, for $\zeta$ in 
$\Hom_{kE}(\gam_\fp k, \gam_\fp k)$, the module $K_\zeta^{\uparrow G}$ is also
dualizable. Moreover, by the Mackey Theorem, any such module has 
at most a finite number of indecomposable direct summands. Now the 
theorem follows from the fact that, by Theorem~\ref{th:groups},  there is an uncountable number
of such modules and they are in $\gam_\fp \StMod(kG)$. On the other hand,
the thick subcategory obtained by taking the idempotent completion
of $\gam_\fp \stmod(kG)$ has only a countable number of 
indecomposable objects. 
\end{proof}

\end{document}